\documentclass[a4paper]{amsart}

\usepackage{esint}

\usepackage{amssymb}
\usepackage{amsmath}
\usepackage{mathrsfs}

\usepackage{graphicx}

\usepackage[hidelinks]{hyperref}
\usepackage[nameinlink]{cleveref} 

\def\restrict#1{\raise-.5ex\hbox{\ensuremath|}_{#1}}
\newcommand{\E}{\mathbb{E}}

\newtheorem{theorem}{Theorem}[section]

\newtheorem{lemma}[theorem]{Lemma}

\theoremstyle{definition}

\theoremstyle{remark}

\usepackage[backrefs]{amsrefs}

\renewcommand{\P}{\mathbb{P}}

\newcommand{\set}[1]{\left\{#1\right\}}

\newcommand{\norm}[1]{\left\Vert#1\right\Vert}

\newcommand{\intersect}{\cap}

\title{Two Approximation Results for Divergence Free Measures}

\author{Jesse Goodman}
\address[J. Goodman]{Department of Statistics, University of Auckland, Private Bag 92019, Auckland 1142, New Zealand}
\email[J. Goodman]{jesse.goodman@auckland.ac.nz}
\thanks{} 

\author{Felipe Hernandez}
\address[F. Hernandez]{Department of Mathematics,
Building 380, Stanford, California 94305, USA}
\email[F. Hernandez]{felipehb@stanford.edu}

\author{Daniel Spector}
\address[D. Spector]{Department of Mathematics, National Taiwan Normal University, No. 88, Section 4, Tingzhou Road, Wenshan District, Taipei City, Taiwan 116, R.O.C.
}
\email[D. Spector]{spectda@protonmail.com}

\begin{document}

\renewcommand{\sectionautorefname}{Section} 

\begin{abstract} 
In this paper we prove two approximation results for divergence free measures.  
The first is a form of an assertion of J.\ Bourgain and H.\ Brezis concerning the approximation of solenoidal charges in the strict topology:  
Given $F \in M_b(\mathbb{R}^d;\mathbb{R}^d)$ such that $\operatorname*{div} F=0$ in the sense of distributions, there exist oriented $C^1$ loops $\Gamma_{i,l}$ with associated measures $\mu_{\Gamma_{i,l}}$ such that
\begin{align*}
F= \lim_{l \to \infty} \frac{\|F\|_{M_b(\mathbb{R}^d;\mathbb{R}^d)}}{n_l \cdot l} \sum_{i=1}^{n_l}  \mu_{\Gamma_{i,l}}
\end{align*}
weakly-star in the sense of measures and
\begin{align*}
\lim_{l \to \infty} \frac{1}{n_l \cdot l} \sum_{i=1}^{n_l}  \|\mu_{\Gamma_{i,l}}\|_{M_b(\mathbb{R}^d;\mathbb{R}^d)} = 1.
\end{align*}
The second, which is an almost immediate consequence of the first, is that smooth compactly supported functions are dense in 
\[
\left\{ F \in M_b(\mathbb{R}^d;\mathbb{R}^d): \operatorname*{div}F=0 \right\}
\]
with respect to the strict topology.
\end{abstract}

\maketitle

\section{Main Results and Discussion}
In this paper, we prove two results concerning the approximation of divergence free measures.  
We explain how these results relate to other recent developments involving the dimension of measures with differential constraints and estimates for elliptic systems.  

\subsection{Main Results}
To state our first result, we note that for a piecewise $C^1$ curve $\Gamma\subset\mathbb{R}^d$ parametrized by arc length via $\gamma\colon[0,l]\to\mathbb{R}^d$ with $|\dot\gamma(t)|=1$, the mapping  
\begin{align*}
C_0(\mathbb{R}^d;\mathbb{R}^d) \to \mathbb{R}, \qquad \Phi\mapsto \int_0^l  \Phi(\gamma(t))\cdot \dot\gamma(t) \, dt,
\end{align*}
is a bounded linear functional on $C_0(\mathbb{R}^d;\mathbb{R}^d)$.  
By the Riesz representation theorem we can identify $\Gamma$ with a finite Radon measure $\mu_\Gamma \in M_b(\mathbb{R}^d;\mathbb{R}^d)$ characterized by
\begin{align*}
\int_{\mathbb{R}^d} \Phi\cdot d\mu_\Gamma = \int_0^l  \Phi(\gamma(t))\cdot \dot\gamma(t) \, dt 
\end{align*}
for all $\Phi\in C_0(\mathbb{R}^d,\mathbb{R}^d)$.  
We also recall that the distributional divergence of $F\in M_b(\mathbb{R}^d;\mathbb{R}^d)$ is characterized by the formula
\begin{align*}
\langle \operatorname*{div} F,\varphi \rangle := -\int_{\mathbb{R}^d} \nabla \varphi \cdot dF
\end{align*}
for all $\varphi \in C^1_c(\mathbb{R}^d)$.  


\begin{theorem}\label{bbassertion}
Suppose $F \in M_b(\mathbb{R}^d;\mathbb{R}^d)$ is such that $\operatorname*{div} F=0$ in the sense of distributions.  
Then there exist oriented $C^1$ closed curves $\Gamma_{i,l}$ with associated measures $\mu_{\Gamma_{i,l}}$ such that
\begin{align*}
F= \lim_{l \to \infty} \frac{\|F\|_{M_b(\mathbb{R}^d;\mathbb{R}^d)}}{n_l \cdot l} \sum_{i=1}^{n_l}  \mu_{\Gamma_{i,l}}
\end{align*}
weakly-star in the sense of measures and
\begin{align*}
\lim_{l \to \infty} \frac{1}{n_l \cdot l} \sum_{i=1}^{n_l}  \|\mu_{\Gamma_{i,l}}\|_{M_b(\mathbb{R}^d;\mathbb{R}^d)} = 1.
\end{align*}
\end{theorem}

A $C^1$ closed curve $\Gamma$ naturally yields a divergence free Radon measure.
Indeed, for $\varphi \in C^1_c(\mathbb{R}^d)$, we compute
\begin{align*}
\langle \operatorname*{div} \mu_\Gamma,\varphi \rangle &= -\int_0^l  \nabla \varphi(\gamma(t))\cdot \dot\gamma(t) \, dt\\
&= -\int_0^l  \frac{d}{dt} \varphi(\gamma(t)) \, dt\\
&= -\varphi(\gamma(l))+\varphi(\gamma(0))\\
&=0.
\end{align*}
Therefore, \autoref{bbassertion} allows one to handle problems concerning the generic case of a divergence free Radon measure with finite mass, provided one can handle the simpler case of $C^1$ closed curves, modulo weak-star convergence.
This has a number of useful applications.
For example, we use \autoref{bbassertion} to prove the following result, which states that smooth compactly supported functions are dense within the space of all divergence free Radon measures.

%
%
%

\begin{theorem}\label{density}
Suppose $F \in M_b(\mathbb{R}^d;\mathbb{R}^d)$ is such that $\operatorname*{div} F=0$ in the sense of distributions.  
Then there exists a sequence of smooth, compactly supported divergence free functions $F_l$ such that 
\begin{align*}
F = \lim_{l \to \infty} F_l
\end{align*}
weakly-star as measures and
\begin{align*}
\lim_{l \to \infty} \|F_l\|_{L^1(\mathbb{R}^d;\mathbb{R}^d)} =  \|F\|_{M_b(\mathbb{R}^d;\mathbb{R}^d)}.
\end{align*}
\end{theorem}
\autoref{density} is one of a class of results stating that functions satisfying a differential constraint can be approximated by smooth compactly supported functions satisfying the same constraint, see e.g. \cite{GuerraRaita}*{Proposition 3.16 on p.~290} or \cite{BonamiPoornima}*{Lemma 1 on p.~177}.
A naive attempt to produce compact support -- multiplying $F$ by a cutoff function -- does not work, as it destroys the differential constraint $\operatorname*{div} F=0$.  
The arguments in \cites{GuerraRaita,BonamiPoornima} use the differential constraint to lift $F$ to another object; apply a cutoff argument to this lifted object; and then project back to $F$.
For example, when the differential constraint from \autoref{density} is instead $\operatorname*{curl}F=0$, A.\ Bonami and S.\ Poornima \cite{BonamiPoornima} lift $F$ to a potential $u \in \dot{W}^{1,1}(\mathbb{R}^d)$ such that $F=\nabla u$.
Even with the vast literature concerning the properties of gradients, the rest of the argument is non-trivial: Bonami and Poornima prove that $W^{1,1}(\mathbb{R}^d)$ is dense in $\dot{W}^{1,1}(\mathbb{R}^d)$, using the boundedness of certain singular integral operators on functions with constrained Fourier support.  
By contrast, \autoref{bbassertion} allows us to prove \autoref{density} using only standard mollification arguments.




 

We continue the introduction with a discussion of the connections with Smirnov's theorem, the dimension of singularities of measures, estimates for elliptic systems, and a further approximation which gives uniformity over the curves before providing proofs of \autoref{bbassertion} and \autoref{density} in Sections \ref{probability} and \ref{proofs}.

\subsection{Discussion}

\subsubsection{Smirnov's Theorem}\label{smirnov}
The basis of \autoref{bbassertion} is a result of S.\ Smirnov \cite{Smirnov}*{Theorem A on p.~847}, quoted here in part.
Write $\mathcal{C}_l$ for the space of rectifiable curves in $\mathbb{R}^d$ of length $l$.
Given $F \in M_b(\mathbb{R}^d;\mathbb{R}^d)$ such that $\operatorname*{div} F=0$ in the sense of distributions, for each $l>0$ there exists a measure $\mu$ on $\mathcal{C}_l$ such that
\begin{equation*}
    \langle F,\Phi\rangle = \int_{\mathcal{C}_l} \langle R,\Phi\rangle \;d\mu(R).
\end{equation*}
Moreover the measure $\mu$ satisfies $\|\mu\|_{M_b(\mathcal{C}_l)} = l^{-1} \|F\|_{M_b(\mathbb{R}^d;\mathbb{R}^d)}$ and 
\[
\frac{ |F|}{l} = \int_{\mathcal{C}_l} \delta_{b(R)} d\mu(R)
    = \int_{\mathcal{C}_l} \delta_{e(R)} d\mu(R)
    ,
\]
where $b(R)$ and $e(R)$ are the beginning and endpoints of the curve $R$; here the total variation measure $|F|$ is the non-negative measure on $\mathbb{R}^d$ defined by 
\begin{align*}
\langle |F|, \varphi \rangle = \sup_{\Phi \in C_c(\mathbb{R}^d; \mathbb{R}^d) , \|\Phi\|_{C_0(\mathbb{R}^d;\mathbb{R}^d)} \leq 1} \langle F ,  \varphi \Phi \rangle
.
\end{align*}

In contrast to Smirnov's theorem, the curves in \autoref{bbassertion} are closed, $C^1$, and need not have length $l$.  
The gain in smoothness is possible because \autoref{bbassertion} is not a decomposition but an approximation, while the change in length is a result of the process of closing the curves.  
This closing of Smirnov's curves yields curves whose lengths may in principle lie anywhere in the interval $[l,2l]$; however, the second convergence assertion of the theorem shows that these lengths are typically of length $l$ in the limit. 
That one approximates a given divergence free measure by closed curves is important for estimates, see \autoref{uniformity_over_curves} and in particular equation \eqref{closed_loop_estimate} below.

\subsubsection{Dimension of Singularities of Measures with Differential Constraints}
The question of the dimension of the space
\begin{align*}
\left\{ F \in M_b(\mathbb{R}^d;\mathbb{R}^k) : LF=0\right\}
\end{align*}
where $L$ is a homogeneous differential or pseudo-differential operator has a long and involved history.  
Here we recall that the Hausdorff dimension of a finite Radon measure is defined as
\begin{align*}
\operatorname*{dim_\mathcal{H}}F:= \sup_{\beta>0} \left\{ \beta : \mathcal{H}^\beta(E)=0 \implies |F|(E)=0\right\}
\end{align*}
where $|F|$ is the total variation measure associated to $F$ defined in the preceding section, while the dimension of a closed subspace $X \subset M_b(\mathbb{R}^d;\mathbb{R}^k)$ can be defined as
\begin{align*}
\kappa:= \inf_{F \in X} \operatorname*{dim_\mathcal{H}}F.
\end{align*}

Smirnov's result \cite{Smirnov}*{Theorem A on p.~847}, Roginskaya and Wojciechowski's \cite{RW}*{Corollary 4 on p.~220}, and our \autoref{bbassertion} are manifestations of the fact that
\begin{align*}
\left\{ F \in M_b(\mathbb{R}^d;\mathbb{R}^d) : \operatorname*{div}F=0\right\}
\end{align*}
has dimension $\kappa=1$.  Indeed, the decompositions provide the lower bound, while the fact that closed curves are divergence free measures gives the upper bound.  
By contrast, the space
\begin{align*}
\left\{ F \in M_b(\mathbb{R}^d;\mathbb{R}^d) : \operatorname*{curl}F=0\right\},
\end{align*}
has $\kappa=d-1$.  This can be seen from the identification
\begin{align*}
F = \nabla u \in \dot{BV}(\mathbb{R}^d),
\end{align*}
whereupon the $BV(\mathbb{R}^d)$ theory yields $\kappa = d-1$, see \cite{AmbFusPal2000}*{Lemma~3.76 on p.~170}.

Similar phenomena also apply for pseudo-differential constraints.  For example, a classical result of F.\ Riesz and M.\ Riesz states that a measure on the circle whose Fourier transform is supported on the positive integers is absolutely continuous with respect to the Lebesgue measure, see e.g. \cite{Havin}*{p.~13}.  Note that $\operatorname*{supp}\widehat{\mu} \subset \mathbb{Z}^+$ is equivalent to $[|n|-n]\widehat{\mu}(n)= 0$, which can be expressed as $L\mu=0$ for $L= (-\Delta)^{1/2} -i \frac{d}{dx}$.  Thus, their result implies that the pseudo-differentially constrained space
\begin{align*}
\left\{ \mu \in M_b(S^1;\mathbb{C}) : L\mu= 0 \right\}
\end{align*}
has dimension $\kappa=1\;(=d)$.

For further results on differential constraints and dimension, we refer the reader to \cites{AW,ARDPHR,AAR,Raita_report,DS,RA, Stol-Woj}.



\subsubsection{From Curves to Divergence Free Measures:  Estimates for Integrals Operators}

The following result was established by the second and third named authors in \cite{HS}.  
\begin{theorem}[Theorem 1.1 in \cite{HS}]\label{mainresult}
Let $d\geq 2$ and $\alpha \in (0,d)$.  There exists a constant $C=C(\alpha,d)>0$ such that
\begin{align}\label{potentialnodiracl1'}
\|I_\alpha F \|_{L^{d/(d-\alpha),1}(\mathbb{R}^d;\mathbb{R}^d)} \leq C \|F\|_{L^1(\mathbb{R}^d;\mathbb{R}^d)}
\end{align}
for all fields $F \in L^1(\mathbb{R}^d;\mathbb{R}^d)$ such that $\operatorname*{div} F=0$ in the sense of distributions.
\end{theorem}
\noindent
Here we use $I_\alpha$ to denote the Riesz potential of order $\alpha \in (0,d)$ (for a precise definition see \cite{Stein}*{p.~117} or \cite{HS}).

The first step in the proof, inspired by \cites{KrantzPelosoSpector,Spector1,Spector2} and a suggestion of Haim Brezis, is to use \autoref{bbassertion} to write $F$ as a weak-star limit of convex combinations of closed rectifiable curves.
This approach is based on H. Brezis and J. Bourgain's assertion \cite{BourgainBrezis2004}*{p.~541} and \cite{BourgainBrezis2007}*{p.~278} that
\begin{align}\label{BB_F_formula}
F= \lim_{l \to \infty}  \sum_{i=1}^{n_l}  \alpha_{i,l} \frac{\mu_{\Gamma_{i,l}}}{\|\mu_{\Gamma_{i,l}}\|_{M_b(\mathbb{R}^d;\mathbb{R}^d)} },
\end{align}
for some choice of closed rectifiable curves $\Gamma_{i,l}$ and scalars $\alpha_{i,l}\geq 0$ which satisfy $\sum_{i=1}^{n_l} \alpha_{i,l} \leq \|F\|_{M_b(\mathbb{R}^d;\mathbb{R}^d)}$.
If we define
\begin{align*}
\alpha_{i,l}:= \frac{\|F\|_{M_b(\mathbb{R}^d;\mathbb{R}^d)} \|\mu_{\Gamma_{i,l}}\|_{M_b(\mathbb{R}^d;\mathbb{R}^d)}}{ \sum_{i=1}^{n_l}  \|\mu_{\Gamma_{i,l}}\|_{M_b(\mathbb{R}^d;\mathbb{R}^d)}},
\end{align*}
then our \autoref{bbassertion} implies \eqref{BB_F_formula} and thus verifies Bourgain and Brezis's assertion \cite{BourgainBrezis2004}*{p.~541} and \cite{BourgainBrezis2007}*{p.~278}, with additional smoothness in the curves.




\subsubsection{Uniformity over Curves}\label{uniformity_over_curves}


\autoref{bbassertion} converts \autoref{mainresult} to the estimate restricted to curves, the inequality
\begin{align}\label{loops}
\|I_\alpha \mu_\Gamma\|_{L^{d/(d-\alpha),1}(\mathbb{R}^d;\mathbb{R}^d)}\leq C'\|\mu_{\Gamma}\|_{M_b(\mathbb{R}^d;\mathbb{R}^d)}
\end{align}
for any smooth, closed curve $\Gamma$.  
That is, one must estimate the fractional integral of a curve $\Gamma$ in a Lorentz space in terms of its length, which be rescaling can be assumed to be one.

Because $\mu_\Gamma$ is an oriented closed loop, there is a minimal surface that spans $\mu_\Gamma$.
An argument based on maximal functions leads to the useful inequality 
\begin{align}\label{closed_loop_estimate}
\|I_\alpha \mu_\Gamma\|_{L^{1,\infty}(\mathbb{R}^d;\mathbb{R}^d)}\leq C\left(\|\mu_{\Gamma}\|_{M_b(\mathbb{R}^d;\mathbb{R}^d)}+\|\mu_{\Gamma}\|^2_{M_b(\mathbb{R}^d;\mathbb{R}^d)}\right),
\end{align}
see \cite{HS}*{Lemma 4.1 and its consequences}.
In order to obtain \eqref{loops},  a second estimate is needed, and the relevant quantity (see \cite{HernandezRaitaSpector2022}*{equation (1.5)} or \cite{HS}*{equation (1.18)}) turns out to be the norm on the Morrey space $\mathcal{M}^1(\mathbb{R}^d)$,
\begin{align*}
\|\mu\|_{\mathcal{M}^1(\mathbb{R}^d)}:= \sup_{r>0,x\in \mathbb{R}^d} \frac{|\mu|(B(x,r))}{r}
 \end{align*}
 for locally finite Radon measures $\mu$.  
The curves provided by \autoref{bbassertion} need not admit a uniform bound on their Morrey norms.
However, because they are curves they lend themselves to further geometric manipulation.  This was the basis for the Surgery Lemma \cite{HS}*{Lemma~5.1}:
\begin{lemma}
    \label{surgery-lem}
Suppose $\Gamma$ is an oriented $C^1$ closed  curve.  There exist oriented piecewise $C^1$ closed curves $\{\Gamma_j\}_{j=1}^{N(\Gamma)}$ with associated measures $\{\mu_{\Gamma_j}\}_{j=1}^{N(\Gamma)}$ such that
    \begin{enumerate}
\item \begin{align*}
\mu_\Gamma= \sum_{j=1}^{N(\Gamma)} \mu_{\Gamma_j};
\end{align*}
    \item The total length of the curves obtained in the decomposition satisfies
\begin{align*}
\sum_{j=1}^N \|\mu_{\Gamma_j}\|_{M_b(\mathbb{R}^d;\mathbb{R}^d)} \leq 10 \|\mu_{\Gamma} \|_{M_b(\mathbb{R}^d;\mathbb{R}^d)};
\end{align*}
    \item
    Each $\mu_{\Gamma_j}$ satisfies the ball growth condition
  \begin{align*}
 \|\mu_{\Gamma_j}\|_{\mathcal{M}^1(\mathbb{R}^d)} = \sup_{x\in \mathbb{R}^d,r>0} \frac{|\mu_{\Gamma_j}|(B(x,r))}{r} \leq 1000.
\end{align*}
    \end{enumerate}
\end{lemma}
\noindent
The combination of \Cref{surgery-lem} and \autoref{bbassertion} shows that any divergence free measure can be approximated by sequences of sums of oriented closed loops with a uniform bound in the Morrey space $\mathcal{M}^1(\mathbb{R}^d)$, \cite{HS}*{Theorem 1.5}.  This allows one to deduce \eqref{loops} and in turn \autoref{mainresult}.

\section{Approximating general integrals by sums}\label{probability}

Smirnov's decomposition \cite{Smirnov}*{Theorem~A} represents a divergence free function in terms of an integral over the space of curves.
We begin by observing that such integrals can be expressed as limits of finite sums.

\begin{theorem}\label{intOverClBySums}
  Let $\mathcal{C}_l$ denote the set of curves of length $l$, equipped with the Borel $\sigma$-algebra $\mathscr{G}$. 
  Suppose that $\mu$ is a finite positive measure on $\mathcal{C}_l$ and let $h_j,j\in\mathbb{N}$, be a sequence of $\mathscr{G}$-measurable functions for which $\int_{\mathcal{C}_l} h_j \, d\mu$ exists. 
  Then there exists a sequence of curves $x_i\in\mathcal{C}_l$, $i\in\mathbb{N}$, such that 
  \begin{equation}\label{NormTimesSumGivesIntegral}
    \lim_{n\to\infty}\frac{\norm{\mu}_{M_b(\mathcal{C}_l)}}{n}\sum_{i=1}^n h_j(x_i)=\int_{\mathcal{C}_l} h_j(x)\,d\mu(x)\quad\text{for all $j\in\mathbb{N}$.}
  \end{equation}
\end{theorem}

The idea in \autoref{intOverClBySums} is that an integral $\int_{\mathcal{C}} h\,d\mu$ over a general space $\mathcal{C}$ can be expressed as a limit of weighted sums: 
\begin{equation}\label{LimitOfWeightedSums}
  \int_{\mathcal{C}} h\,d\mu = \lim_{n\to\infty}\sum_{i=1}^n c_{i,n}h(x_{i,n}),
\end{equation}
where $c_{i,n}$ are suitably chosen scalars and $x_{i,n}\in\mathcal{C}$ are suitably chosen points.

For common choices of space $\mathcal{C}$ we may select the points $x_{i,n}\in\mathcal{C}$ explicitly.
For instance, when $\mathcal{C}$ is a finite interval $[a,b]$, we can choose equally spaced points $x_{i,n}=a+i\frac{b-a}{n}$, with $c_{i,n}=1/n$ for all $i$.
Many other choices are possible: for instance, Simpson's rule for integration (with $n=2k+1$ odd and, for convenience, $i$ running from 0 to $2k$) takes $x_{i,n}=a+i\frac{b-a}{2k}$ and $(c_{0,n},\dotsc,c_{n,n}) = \frac{1}{6k}(1,4,2,4,2\dotsc 2,4,2,4,1)$.

The quantity in \eqref{LimitOfWeightedSums} resembles a Riemann sum approximation to the integral $\int_a^b h(x)\,dx$.
There are however notable differences: Riemann integration requires that the limit in \eqref{LimitOfWeightedSums} should exist when $h(x_{i,n})$ is replaced by the supremum, or infimum, of $h$ over a suitably chosen subinterval to which $x_{i,n}$ belongs, and the limit should exist for any subdivision of $[a,b]$ into small subintervals \cite{RudinPoMA}*{Chapter 6 and Theorem 11.33}.

For a less structured space such as $\mathcal{C}_l$, there may be no natural way to choose points $x_{i,n}$ \textit{a priori}.
We will avoid this difficulty by choosing random points $X_i$.

\begin{proof}[Proof of \autoref{intOverClBySums}]
Normalize the finite measure $\mu$ to produce a probability measure $\nu=\mu/\|\mu\|_{M_b(\mathcal{C}_l)}$ on $(\mathcal{C}_l,\mathscr{G})$.
Construct $\Omega=\mathcal{C}_l^\mathbb{N}$, the set of infinite sequences with values in $\mathcal{C}_l$, equipped with the product $\sigma$-algebra $\mathscr{F}=\mathscr{G}^{\otimes \mathbb{N}}$.
On the measurable space $(\Omega,\mathscr{F})$, assign the product measure $\mathbb{P}=\nu^{\otimes \mathbb{N}}$.
Set $X_i\colon\Omega\to\mathcal{C}_l$ to be the $i^\text{th}$ coordinate function: for a sequence $\omega=(\omega_1,\omega_2,\dotsc)$, set $X_i(\omega)=\omega_i$.
From the definition of product $\sigma$-algebra, the function $X_i\colon \Omega\to\mathcal{C}_l$ is measurable as a mapping from the measurable space $(\Omega,\mathscr{F})$ to the measurable space $(\mathcal{C}_l,\mathscr{G})$.

In probabilistic language, the probability space $(\Omega,\mathscr{F},\mathbb{P})$ corresponds to a random experiment where each point $X_i$ is chosen according to $\P(X_i\in A)=\nu(A)=\mu(A)/\mu(\mathcal{C}_l)$ for any $A\in\mathscr{G}$.
Furthermore, if $A_i\in\mathscr{G}$ for all $i$ then the events $\set{X_i\in A_i}$ are independent across different $i$.
Thus the $X_i$'s are independent and identically distributed (i.i.d.) random variables with values in $\mathcal{C}_l$ and law $\nu$.
(As is standard, this formulation elides the role of the $\sigma$-algebra $\mathscr{G}$ and the underlying probability space $(\Omega,\mathscr{F},\mathbb{P})$.)

For each measurable function $h_j\colon\mathcal{C}_l\to\mathbb{R}$, we can define real-valued random variables $H_{i,j} = h_j(X_i)$.
Then for each fixed $j$, the random variables $H_{i,j}$, $i\in\mathbb{N}$, are themselves i.i.d.\ with common expected value
\begin{align*}
\E(H_{i,j}) = \int_\Omega h_j(X_i(\omega)) \, d\mathbb{P}(\omega) = \int_\Omega h_j(\omega_i) \, d\mathbb{P}(\omega) = \int_{\mathcal{C}_l} h_j(x) \, d\nu(x)
\end{align*}
by the properties of product measure.

In this setting, the Strong Law of Large Numbers, see for instance \cite{Durrett2019}*{Theorem 2.5.10}, asserts that 
\begin{align}\label{SLLNhXi}
\lim_{n\to\infty} \frac{1}{n}\sum_{i=1}^n h_j(X_i) = \int_{\mathcal{C}_l} h_j(x) \, d\nu(x) \quad\text{almost surely}.
\end{align}
More precisely, the function 
\begin{align*}
\omega\mapsto \lim_{n\to\infty} \frac{1}{n}\sum_{i=1}^n h_j(X_i(\omega))
\end{align*}
exists and equals the constant $\int_{\mathcal{C}_l} h_j(x) \, d\nu(x)$ for $\mathbb{P}$-almost-every $\omega$.
In other words, each set 
\begin{equation}
  B_j = \set{\omega\in\Omega\colon\lim_{n\to\infty} \frac{1}{n}\sum_{i=1}^n h_j(X_i(\omega)) = \int_{\mathcal{C}_l} h_j(x) \, d\nu(x)} 
\end{equation}
has $\P(B_j)=1$ and $\P(B_j^c)=0$.
Taking a countable intersection $B=\intersect_{j=1}^\infty B_j$, it follows that $\P(B^c)=0$ and hence $\P(B)=1$.
In particular, $B$ must be non-empty, so there exists some $\tilde{\omega}\in B$.
Defining $x_i=X_i(\tilde{\omega})$, the definition of $B$ implies that 
\begin{equation}
  \lim_{n\to\infty} \frac{1}{n}\sum_{i=1}^n h_j(x_i) = \int_{\mathcal{C}_l} h_j(x) \, d\nu(x) \quad\text{for all $j\in\mathbb{N}$}
\end{equation}
and multiplying both sides by $\norm{\mu}_{M_b(\mathcal{C}_l)}$ yields \eqref{NormTimesSumGivesIntegral}.
\end{proof}

Note that the random curves $X_i$, i.e., the functions $\omega\mapsto X_i(\omega)$, depend neither on $n$ nor on $h$.
However, the proof is non-constructive: the fact that $B$ is non-empty implies the existence of some sequence of curves $x_i$ for which \eqref{LimitOfWeightedSums} holds, but does not give a specific sequence.
In particular, arguments based on \eqref{SLLNhXi} must contend with the fact that \eqref{LimitOfWeightedSums} holds only almost everywhere, and the exceptional set $\Omega\setminus B$ (and hence the chosen points $x_i$) may \textit{a priori} depend on the choice of functions $h_j$.

The quantity inside the limit in \eqref{SLLNhXi} can be interpreted as the integral of $h_j$ with respect to a random measure: if we define a measure $\eta_n$ on $\mathcal{C}_l$ by
\begin{align*}
\eta_n = \frac{1}{n}\sum_{i=1}^n \delta_{X_i}
\end{align*}
(with $\delta_x$ denoting the Dirac mass at $x\in\mathcal{C}_l$) then
\begin{align*}
\int_{\mathcal{C}_l} h_j(x) \, d\eta_n(x) = \frac{1}{n}\sum_{i=1}^n h_j(X_i).
\end{align*}
Since $\mathcal{C}_l$ has additional structure, it is possible to argue that $\eta_n$ converges $\P$-a.s.\ to $\nu$ in the weak topology for measures on $\mathcal{C}_l$.
In this case, \eqref{SLLNhXi} holds simultaneously for all continuous bounded functions $h$, a.s., with a single exceptional set of measure zero for all such functions $h$.
Specifically, this will occur if we can find a countable collection of functions $h_j$ that are convergence-determining for the weak topology for measures on $\mathcal{C}_l$.
This is the case in the proof of \autoref{lln-cor} below, though since $\nu$ is not our primary focus we will carry out this part of the argument for $F$ rather than for $\nu$.

\section{Proofs} \label{proofs}

As we now explain, \autoref{intOverClBySums} and Smirnov's decomposition allow us to represent a divergence free function $F$ in terms of a sequence of curves.
In the remainder of the paper, we denote curves using the letter $R$ instead of $x$ as in \autoref{intOverClBySums}.

By Theorem A in \cite{Smirnov} we have
\begin{equation}
    \label{smirnov-appx}
    \langle F,\Phi\rangle = \int_{\mathcal{C}_l} \langle R,\Phi\rangle \;d\mu(R),
\end{equation}
where the measure $\mu$ satisfies $\|\mu\|_{M_b(\mathcal{C}_l)} = l^{-1} \|F\|_{M_b(\mathbb{R}^d;\mathbb{R}^d)}$ and
\[
\frac{ |F|}{l} = \int_{\mathcal{C}_l} \delta_{b(R)} d\mu(R)
    = \int_{\mathcal{C}_l} \delta_{e(R)} d\mu(R)
\]
where $b(R)$ and $e(R)$ are the beginning and endpoints of the curve $R$ and the total variation measure $|F|$ is the non-negative measure on $\mathbb{R}^d$ defined by 
\begin{align*}
\langle |F|, \varphi \rangle = \sup_{\Phi \in C_c(\mathbb{R}^d; \mathbb{R}^d) , \|\Phi\|_{C_0(\mathbb{R}^d;\mathbb{R}^d)} \leq 1} \langle F ,  \varphi \Phi \rangle
\end{align*}

This implies the auxiliary 
\begin{theorem}
    \label{lln-cor}
    For any $l>0$ there exists a sequence of curves $R_i\in\mathcal{C}_l$
    satisfying the following conditions:
    For any $\Phi\in C_0(\mathbb{R}^d;\mathbb{R}^d)$
    \[
        \langle F, \Phi\rangle
        = \lim_{n\to\infty}
        \frac{\|F\|_{M_b(\mathbb{R}^d;\mathbb{R}^d)}}{n} \sum_{i=1}^n \frac{\langle R_i,\Phi\rangle}{l}.
    \]
    Moreover, with $b(R)$ and $e(R)$ denoting the beginning and end points
    of the curve $R$, we have for any $\varphi\in C_0(\mathbb{R}^d)$
    \begin{equation}
        \label{|F|-appx}
    \langle |F|, \varphi\rangle = \lim_{n\to\infty} \frac{\|F\|_{M_b(\mathbb{R}^d;\mathbb{R}^d)}}{n} \sum_{i=1}^n \langle \delta_{b(R_i)}, \varphi\rangle 
   = \lim_{n\to\infty} \frac{\|F\|_{M_b(\mathbb{R}^d;\mathbb{R}^d)}}{n} \sum_{i=1}^n \langle \delta_{e(R_i)}, \varphi\rangle.
    \end{equation}
\end{theorem}

\begin{proof}[Proof of \autoref{lln-cor} using \autoref{intOverClBySums}]
    Let $\nu = l \|F\|_{M_b(\mathbb{R}^d;\mathbb{R}^d)}^{-1} \mu$ be the measure obtained by scaling the measure $\mu$ in~\eqref{smirnov-appx}.  
    Let $\{\Phi_j\}_{j=1}^\infty$ and $\{\varphi_j\}_{j=1}^\infty$ be dense sequences of functions in $C_0(\mathbb{R}^d;\mathbb{R}^d)$ and $C_0(\mathbb{R}^d)$, respectively.  
    For $j\in \mathbb{N}$, we define the continuous functions
    $h_j:\mathcal{C}_l\to\mathbb{R}$,
    $h^b_j:\mathcal{C}_l\to\mathbb{R}$, and
    $h^e_j:\mathcal{C}_l\to\mathbb{R}$
    by
    \begin{align*}
        h_j(R) &:= \langle R, \Phi_j\rangle, &
        h^b_j(R) &:= \langle \delta_{b(R)}, \varphi_j\rangle, &
        h^e_j(R) &:= \langle \delta_{e(R)}, \varphi_j\rangle.
    \intertext{Note that the inequalities}
|h_j(R)|&\leq \|\Phi_j\|_{C_0(\mathbb{R}^d;\mathbb{R}^d)} l, &
|h^b_j(R)| &\leq \|\varphi_j\|_{C_0(\mathbb{R}^d)}, &
|h^e_j(R)| &\leq \|\varphi_j\|_{C_0(\mathbb{R}^d)}
\intertext{imply the moment conditions}
 \int_{\mathcal{C}_l} &|h_j|\;d\nu < \infty, &
 \int_{\mathcal{C}_l} &|h^b_j|\;d\nu < \infty, &
 \int_{\mathcal{C}_l} &|h^e_j|\;d\nu < \infty.
 \end{align*}
  Applying \autoref{intOverClBySums} with the interleaved sequence of functions $h_1,h^b_1,h^e_1,h_2,h^b_2,\dotsc$, we obtain a sequence of curves $R_i,i\in\mathbb{N}$, such that 
\begin{align*}
  \langle F, \Phi_j\rangle &=    \frac{\|F\|_{M_b(\mathbb{R}^d;\mathbb{R}^d)}}{l} \lim_{n\to\infty} \frac{1}{n} \sum_{i=1}^n
  \langle R_i, \Phi_j\rangle,
  \\
  \langle |F|, \varphi_j\rangle  &= \lim_{n\to\infty} \frac{\|F\|_{M_b(\mathbb{R}^d;\mathbb{R}^d)}}{n} \sum_{i=1}^n \langle \delta_{b(R_i)}, \varphi_j\rangle = \lim_{n\to\infty} \frac{\|F\|_{M_b(\mathbb{R}^d;\mathbb{R}^d)}}{n}
  \sum_{i=1}^n \langle \delta_{e(R_i)}, \varphi_j\rangle
\end{align*}
for all $j\in\mathbb{N}$.
For arbitrary $\Phi \in C_0(\mathbb{R}^d;\mathbb{R}^d)$, we utilize the equalities
  \begin{align*}
       \langle F, \Phi \rangle &=  \langle F, \Phi_j\rangle + \langle F, \Phi-\Phi_j\rangle, \\
    \langle R,\Phi \rangle &= \langle R,\Phi_j\rangle + \langle R,\Phi -\Phi_j\rangle,
 \end{align*}
to write
\begin{align*}
        \langle F, \Phi \rangle &=  \langle F, \Phi-\Phi_j\rangle +  \frac{\|F\|_{M_b(\mathbb{R}^d;\mathbb{R}^d)}}{l} \lim_{n\to\infty} \frac{1}{n} \sum_{i=1}^n
        \langle R_i, \Phi\rangle +  \frac{\|F\|_{M_b(\mathbb{R}^d;\mathbb{R}^d)}}{l} \lim_{n\to\infty} \frac{1}{n} \sum_{i=1}^n
        \langle R_i, \Phi_j-\Phi\rangle.
        \end{align*}
Then the bounds 
\begin{align*}
\left|\langle F, \Phi-\Phi_j\rangle \right| &\leq \|F\|_{M_b(\mathbb{R}^d;\mathbb{R}^d)} \|\Phi-\Phi_j\|_{C_0(\mathbb{R}^d;\mathbb{R}^d)} \\
\left|\langle R,\Phi -\Phi_j\rangle\right| &\leq l \|\Phi-\Phi_j\|_{C_0(\mathbb{R}^d;\mathbb{R}^d)}
\end{align*}
imply
\begin{align*}
        \langle F, \Phi \rangle &=   \frac{\|F\|_{M_b(\mathbb{R}^d;\mathbb{R}^d)}}{l} \lim_{n\to\infty} \frac{1}{n} \sum_{i=1}^n
        \langle R_i, \Phi\rangle + O\left(\|F\|_{M_b(\mathbb{R}^d;\mathbb{R}^d)} \|\Phi-\Phi_j\|_{C_0(\mathbb{R}^d;\mathbb{R}^d)} \right),
        \end{align*}
and it suffices to consider a subsequence such that $\Phi_{j_k} \to \Phi$.  The argument for the other two limits is similar.
\end{proof}

We now prove \autoref{bbassertion}.
\begin{proof}[Proof of \autoref{bbassertion}]

Let $F \in M_b(\mathbb{R}^d;\mathbb{R}^d)$ be such that $\operatorname*{div}F=0$, and by scaling let us assume $\|F\|_{M_b(\mathbb{R}^d;\mathbb{R}^d)}=1$.  
By \autoref{lln-cor} there exists a sequence of curves $R_{i,l}\in\mathcal{C}_l$ such that
\begin{align*}
\langle F, \Phi \rangle= \lim_{n \to \infty} \frac{1}{n} \sum_{i=1}^n \frac{1}{l} \langle R_{i,l}, \Phi \rangle.
\end{align*}
Let us write $\tilde{R}_{i,l}$ for the measure which consists of a closed loop formed by adjoining to $R_{i,l}$ the straight line segment connecting the end point $e(R_{i,l})$ to the beginning point $b(R_{i,l})$.
Write $\overline{R}_{i,l}$ for the measure which is integration along the straight line segment in reverse, from beginning to end.  
Then the preceding result may be rewritten as
\begin{align*}
\langle F, \Phi \rangle = \lim_{n \to \infty} \frac{1}{n} \sum_{i=1}^n \frac{1}{l} \langle \tilde{R}_{i,l}, \Phi \rangle + \lim_{n \to \infty} \frac{1}{n} \sum_{i=1}^n \frac{1}{l} \langle \overline{R}_{i,l}, \Phi \rangle.
\end{align*}

We next show that
\begin{align}\label{claim}
\limsup_{l \to \infty} \limsup_{n \to \infty} \frac{1}{n} \sum_{i=1}^n \frac{1}{l}  \|\overline{R}_{i,l}\|_{M_b(\mathbb{R}^d;\mathbb{R}^d)}=0.
\end{align}
To this end, recall that from the preceding Theorem that $b(R), e(R)$ denote the beginning and ending of the curve $R$, we can write
\begin{align*}
\frac{1}{n} \sum_{i=1}^n \frac{1}{l}  \|\overline{R}_{i,l}\|_{M_b(\mathbb{R}^d;\mathbb{R}^d)} = \frac{1}{n} \sum_{|b(R_{i,l})-e(R_{i,l})| \leq \epsilon l} \frac{1}{l}  \|\overline{R}_{i,l}\|_{M_b(\mathbb{R}^d;\mathbb{R}^d)} +  \frac{1}{n} \sum_{|b(R_{i,l})-e(R_{i,l})|>\epsilon l} \frac{1}{l}  \|\overline{R}_{i,l}\|_{M_b(\mathbb{R}^d;\mathbb{R}^d)}.
\end{align*}
For the first term, we can estimate by the length of the curve to obtain the bound
\begin{align*}
 \frac{1}{n} \sum_{|b(R_{i,l})-e(R_{i,l})|\leq \epsilon l} \frac{1}{l}  \|\overline{R}_{i,l}\|_{M_b(\mathbb{R}^d;\mathbb{R}^d)}  \leq \epsilon.
\end{align*}
Meanwhile, for the second term we have that $ \|\overline{R}_{i,l}\|_{M_b(\mathbb{R}^d;\mathbb{R}^d)} \leq l$ and so
\begin{align*}
 \frac{1}{n} \sum_{|b(R_{i,l})-e(R_{i,l})|>\epsilon l} \frac{1}{l}  \|\overline{R}_{i,l}\|_{M_b(\mathbb{R}^d;\mathbb{R}^d)} \leq \frac{\# \{R_i : |b(R_{i,l})-e(R_{i,l})|>\epsilon l \}}{n}.
 \end{align*}
As
\begin{align*}
\{R_i : |b(R_{i,l})-e(R_{i,l})|>\epsilon l \} \subset \{R_i :b(R_{i,l}) \in B(0,\epsilon l/2)^c  \} \cup \{R_i :e(R_{i,l}) \in B(0,\epsilon l/2)^c  \},
\end{align*}
we can bound the second term by
\begin{align*}
\frac{\# \{R_i :b(R_{i,l}) \in B(0,\epsilon l/2)^c  \}}{n} +  \frac{\# \{R_i :e(R_{i,l}) \in B(0,\epsilon l/2)^c  \}}{n}.
\end{align*}
We claim that the double limit in $l$ and $n$ of this quantity converges to zero.  To this end, we let $\varphi \in C_c(\mathbb{R}^d)$ be a cutoff function, i.e. $0 \leq \varphi \leq 1$, $\operatorname*{supp} \varphi \subset B(0,\epsilon l/2)$, and $\varphi \equiv 1$ on $B(0,\epsilon l/2-1)$.  For such a function we see that, for $l$ sufficiently large,
\begin{align*}
|F|(B(0,\epsilon l/2-1)) \leq \langle |F|, \varphi \rangle = \lim_{n \to \infty} \frac{1}{n} \sum_{i=1}^n \langle \delta_{b(R_{i,l})} \varphi \rangle  \leq \liminf_{n \to \infty} \frac{\# \{R_i :b(R_{i,l}) \in B(0,\epsilon l/2)  \}}{n}.
\end{align*}
In particular,
\begin{align*}
1 = \lim_{l \to \infty} |F|(B(0,\epsilon l/2-1))  \leq \liminf_{l \to \infty}\; \liminf_{n \to \infty} \frac{\# \{R_i :b(R_{i,l}) \in B(0,\epsilon l/2)  \}}{n},
\end{align*}
and therefore 
\begin{align*}
 \limsup_{l \to \infty}\limsup_{n \to \infty}\frac{\# \{R_i :b(R_{i,l}) \in B(0,\epsilon l/2)^c  \}}{n} \leq 1- \liminf_{l \to \infty}\; \liminf_{n \to \infty}\frac{\# \{R_i :b(R_{i,l}) \in B(0,\epsilon l/2)  \}}{n} = 0
\end{align*}
Thus we have shown that
\begin{align*}
\limsup_{l \to \infty}\; \limsup_{n \to \infty} \frac{1}{n} \sum_{i=1}^n \frac{1}{l}  \|\overline{R}_{i,l}\|_{M_b(\mathbb{R}^d;\mathbb{R}^d)} \leq \epsilon
\end{align*}
and it suffices to send $\epsilon$ to zero and the claim is proved.

As a result of \eqref{claim} we have, firstly, the weak convergence
\begin{align*}
\langle F, \Phi \rangle = \lim_{l\to \infty} \lim_{n \to \infty} \frac{1}{n} \sum_{i=1}^n \frac{1}{l} \langle \tilde{R}_{i,l}, \Phi \rangle,
\end{align*}
and secondly, the estimate
\begin{align*}
 \left| \frac{1}{n} \sum_{i=1}^n \frac{1}{l}  \langle \tilde{R}_{i,l}, \Phi \rangle \right| &\leq \frac{1}{n} \sum_{i=1}^n \frac{1}{l}  \|R_{i,l}\|_{M_b(\mathbb{R}^d;\mathbb{R}^d)} + \frac{1}{n} \sum_{i=1}^n \frac{1}{l} \|\overline{R}_{i,l}\|_{M_b(\mathbb{R}^d;\mathbb{R}^d)} \\
 &\leq 1+\frac{1}{n} \sum_{i=1}^n \frac{1}{l} \|\overline{R}_{i,l}\|_{M_b(\mathbb{R}^d;\mathbb{R}^d)}.
\end{align*}
This shows convergence in the strict topology of measures. 

As $\tilde{R}_{i,l} $ are one dimensional rectifiable currents without boundary, we have that $\tilde{R}_{i,l} \in \mathbb{I}_1(\mathbb{R}^d)$.  Therefore for each $\tilde{R}_{i,l}$ we can apply \cite{Federer}*{4.2.20} to obtain a family of one dimensional polygonal chains $P_{i,l,\eta}$ and a family of Lipschitz maps $f^\eta$ for which
\begin{align*}
\lim_{\eta \to 0} \|P_{i,l,\eta}-f_\#^\eta\tilde{R}_{i,l}\|_{M_b(\mathbb{R}^d;\mathbb{R}^d)} = 0.
\end{align*}
The fact that $\tilde{R}_{i,l}$ are without boundary implies that the $P_{i,l,\eta}$ obtained in the theorem are without boundary.  Moreover, the above convergence, the weak-star convergence $f_\#^\eta\tilde{R}_{i,l} \overset{*}{\rightharpoonup}\tilde{R}_{i,l}$, and the bound
\begin{align*}
Lip(f^\eta) \leq 1+\eta
\end{align*}
shows
\begin{align*}
\lim_{\eta \to 0} \|P_{i,l,\eta}\| = \lim_{\eta \to 0} \|f_\#^\eta\tilde{R}_{i,l}\|_{M_b(\mathbb{R}^d;\mathbb{R}^d)} = \|\tilde{R}_{i,l}\|_{M_b(\mathbb{R}^d;\mathbb{R}^d)},
\end{align*}
i.e. the measures $P_{i,l,\eta}$ converge to the measure $\tilde{R}_{i,l}$ in the strict topology.  It only remains to smooth the corners, replacing $P_{i,l,\eta}$ with $\tilde{R}_{i,l,\eta}$ which are closed and decrease the length for each $\eta$, as depicted in the following figure.
\begin{figure}[h]
\centering
\includegraphics[width = 0.95 \textwidth]{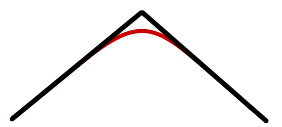}
\caption{A depiction of the smoothing of corners.}
\label{figureSmoothing}
\end{figure}

As the decrease in length can be made to go to zero as $\eta \to 0$, these $\tilde{R}_{i,l,\eta}$ also converge to $\tilde{R}_{i,l}$ in the strict topology, i.e. weak-star convergence
\begin{align*}
\langle F, \Phi \rangle = \lim_{l\to \infty} \lim_{n \to \infty} \lim_{\eta \to 0} \frac{1}{n} \sum_{i=1}^n \frac{1}{l} \langle \tilde{R}_{i,l,\eta}, \Phi \rangle
\end{align*}
and an upper bound for the total variations
\begin{align*}
\limsup_{l\to \infty} \limsup_{n \to \infty} \limsup_{\eta \to 0} \frac{1}{n} \sum_{i=1}^n \frac{1}{l} \|\tilde{R}_{i,l,\eta}\|_{M_b(\mathbb{R}^d;\mathbb{R}^d)}  \leq 1,
\end{align*}
which follows from the upper bound for $P_{i,l,\eta}$ and the decrease in length in their smoothing to $\tilde{R}_{i,l,\eta}$.  From this a diagonal argument yields
\begin{align*}
\langle F, \Phi \rangle = \lim_{l\to \infty}  \frac{1}{n_l} \sum_{i=1}^{n_l} \frac{1}{l} \langle \tilde{R}_{i,l,\eta_l}, \Phi \rangle
\end{align*}
and
\begin{align*}
\limsup_{l\to \infty} \frac{1}{n_l} \sum_{i=1}^{n_l} \frac{1}{l} \|\tilde{R}_{i,l,\eta_l}\|_{M_b(\mathbb{R}^d;\mathbb{R}^d)}  \leq 1.
\end{align*}
The former limit is precisely the weak-star convergence of the convex sum of loops claimed, while it implies
\begin{align*}
\|F\|_{M_b(\mathbb{R}^d;\mathbb{R}^d)} \leq \liminf_{l\to \infty} \frac{1}{n_l} \sum_{i=1}^{n_l} \frac{1}{l} \|\tilde{R}_{i,l,\eta_l}\|_{M_b(\mathbb{R}^d;\mathbb{R}^d)},
\end{align*}
as the total variation is lower semicontinuous with respect to the weak-star convergence.  Thus, when combined with the latter inequality, using $\|F\|_{M_b(\mathbb{R}^d;\mathbb{R}^d)}=1$ we obtain the convergence of total variations claimed.
 
It only remains to adapt the notation to match the statement of the theorem.  Observe that we have found an approximation in terms of smooth curves $\tilde{R}_{i,l,\eta_l}$, which we identify with the Radon measures they induce and denote by $\mu_{\Gamma_{i,l}}:=\tilde{R}_{i,l,\eta_l}$.  Then our result in this notation reads
\begin{align*}
F= \lim_{l \to \infty} \frac{\|F\|_{M_b(\mathbb{R}^d;\mathbb{R}^d)}}{n_l \cdot l} \sum_{i=1}^{n_l}  \mu_{\Gamma_{i,l}}
\end{align*}
weakly-star in the sense of measures and
\begin{equation*}
\lim_{l \to \infty} \frac{1}{n_l \cdot l} \sum_{i=1}^{n_l}  \|\mu_{\Gamma_{i,l}}\|_{M_b(\mathbb{R}^d;\mathbb{R}^d)} = 1.
\qedhere
\end{equation*}
\end{proof}

We conclude by proving \autoref{density}.

\begin{proof}[Proof of \autoref{density}]
Let $\Gamma_{i,l}, l\in\mathbb{N}, i=1,\dotsc,n_l$, be the smooth, closed loops given by \autoref{bbassertion} for which
\begin{align*}
F= \lim_{l \to \infty} \frac{\|F\|_{M_b(\mathbb{R}^d;\mathbb{R}^d)}}{n_l \cdot l} \sum_{i=1}^{n_l}  \mu_{ \Gamma_{i,l}}
\end{align*}
and
\begin{align*}
\lim_{l \to \infty} \frac{1}{n_l \cdot l} \sum_{i=1}^{n_l}  \|\mu_{\Gamma_{i,l}}\|_{M_b(\mathbb{R}^d;\mathbb{R}^d)} = 1.
\end{align*}

Denote by $G_l$ the $l^\text{th}$ approximation of  $F$ by loops, i.e.
\begin{align*}
G_l:= \frac{\|F\|_{M_b(\mathbb{R}^d;\mathbb{R}^d)}}{n_l \cdot l} \sum_{i=1}^{n_l}  \mu_{ \Gamma_{i,l}}.
\end{align*}
Then if $\{\rho_k\}_{k \in \mathbb{N}}$ is a smooth, compactly supported approximation of the identity, we claim $F_l:=G_l \ast \rho_l$ has the desired properties.  

In particular, $F_l$ is smooth by properties of $\rho_l$, compactly supported by the compact support of the loops and $\rho_l$, and for any $\Phi \in C_0(\mathbb{R}^d;\mathbb{R}^d)$ we have
\begin{align}\label{f_l}
\langle F_l, \Phi \rangle= \langle G_l \ast \rho_l, \Phi \rangle = \langle G_l  , \Phi \ast \rho_l \rangle.
\end{align}
This shows firstly that $F_l$ is divergence free, since if $\Phi = \nabla \varphi$ for some $\varphi \in C^1_c(\mathbb{R}^d)$, the fact that derivatives commute with convolution implies that
\begin{align*}
\langle F_l, \nabla \varphi \rangle =\langle G_l, \nabla (\varphi \ast \rho_l) \rangle = -\langle \operatorname*{div} G_l,  \varphi \ast \rho_l \rangle = 0,
\end{align*}
as $\varphi \ast \rho_l \in C^\infty_c(\mathbb{R}^d)$.  Toward the convergence, letting $l \to \infty$ in \eqref{f_l}, utilizing that $\Phi \ast \rho_l \to \Phi$ in the strong topology of $C_0(\mathbb{R}^d;\mathbb{R}^d)$ and $G_l \to F$ weakly-star, we obtain
\begin{align*}
\lim_{l \to \infty} \langle F_l \ast \rho_l, \Phi \rangle =  \langle F , \Phi \rangle,
\end{align*}
which is to say that the sequence $\{F_l\} \subset C^\infty_c(\mathbb{R}^d;\mathbb{R}^d)$ converges to $F$ in the weak-star topology. As this implies
\begin{align*}
\|F\|_{M_b(\mathbb{R}^d;\mathbb{R}^d)} \leq \liminf_{l \to \infty} \|F_l\|_{M_b(\mathbb{R}^d;\mathbb{R}^d)},
\end{align*}
it only remains to show that
\begin{align*}
 \limsup_{l \to \infty} \|F_l\|_{M_b(\mathbb{R}^d;\mathbb{R}^d)} \leq \|F\|_{M_b(\mathbb{R}^d;\mathbb{R}^d)}
\end{align*}
to obtain the strict convergence.  However, Fubini's theorem and the fact that $\int \rho_l=1$ implies
\begin{align*}
\|F_l\|_{M_b(\mathbb{R}^d;\mathbb{R}^d)}  \leq \int_{\mathbb{R}^d} (G_l)_{TV} \ast \rho_l \;dx \leq \|G_l\|_{M_b(\mathbb{R}^d;\mathbb{R}^d)},
\end{align*}
and since
\begin{align*}
\limsup_{l \to \infty}  \|G_l\|_{M_b(\mathbb{R}^d;\mathbb{R}^d)} \leq \lim_{l \to \infty} \frac{\|F\|_{M_b(\mathbb{R}^d;\mathbb{R}^d)}}{n_l \cdot l} \sum_{i=1}^{n_l}  \|\mu_{\Gamma_{i,l}}\|_{M_b(\mathbb{R}^d;\mathbb{R}^d)} = \|F\|_{M_b(\mathbb{R}^d;\mathbb{R}^d)},
\end{align*}
the result is demonstrated.
\end{proof}

\section*{Acknowledgements}
J.G. is supported in part by grants from the Marsden Fund administered by the Royal Society of New Zealand.
F.H. is supported by the Fannie and John Hertz Foundation.
D.S. is supported by the National Science and Technology Council of Taiwan under research grant number 110-2115-M-003-020-MY3 and the Taiwan Ministry of Education under the Yushan Fellow Program.


\bibliography{bib}

\end{document}